\documentclass[12pt]{article}

\usepackage{graphicx}
\usepackage{amsfonts,amsmath,amssymb,amsthm}
\usepackage{newtxtext}
\usepackage{xcolor,soul}
\usepackage{fullpage,comment}
\usepackage{hyperref}

\newtheorem{theorem}{Theorem}
\newtheorem{corollary}[theorem]{Corollary}
\newtheorem{proposition}[theorem]{Proposition}
\newtheorem{lemma}[theorem]{Lemma}
\newtheorem{conjecture}[theorem]{Conjecture}

\newtheorem{remark}{Remark}

\newcommand{\inj}{\mathrm{inj}}

\newcommand{\Fq}{\mathbb{F}_q}

\title{Combinatorial Bounds for Codes over Metric Spaces: Ramsey-Sidorenko Thresholds and Subgraph Counts}

\author{
 Lucas Waite $^{1}$ \quad Nuh Aydin$^{1,}\footnote{Corresponding author: aydinn@kenyon.edu}$\\[1em]  
{\small $^{1}$Department of Mathematics and Statistics, Kenyon College, Gambier, OH 43022, USA}\\
{\small Emails: waite1@kenyon.edu, aydinn@kenyon.edu}\\
}

\date{July 28, 2026}

\begin{document}

\maketitle

\begin{abstract}
This paper investigates the relationship between coding theory and extremal combinatorics by representing codes in general metric spaces as independent sets in proximity graphs. We provide a generalized framework for the Gilbert-Varshamov (GV) bound applicable to codes over any finite metric space and explore the conditions under which global combinatorial parameters can force the existence of codes exceeding this bound. Central to our analysis is the introduction of Ramsey-Sidorenko and independence-forcing graphs. We establish density thresholds for various graph families and utilize the Karush--Kuhn--Tucker conditions to analyze entropy optimization in the Hamming case. Furthermore, we derive upper bounds on code sizes using fractional packings in vertex-transitive and nonedge-transitive graphs. Our findings demonstrate that local subgraph statistics alone are insufficient to surpass the GV bound in the Hamming case, suggesting that improvements must stem from large-scale structural properties of the space.
\end{abstract}

\section{Introduction}
A code $C$ of length $n$ over  $\Fq$ is a subset of $\Fq^n$.  The elements of $C$ are called codewords, the set $\Fq$ is called the (code) alphabet. The most important alphabet in coding theory is $\mathbb{F}_2$. A code over  $\mathbb{F}_2$ is called a binary code, a code over  $\mathbb{F}_3$ is called a ternary code, and a code over $\Fq$ is called a $q$-ary code. The Hamming distance $d_H(\mathbf{u}, \mathbf{v})$ between two vectors $\mathbf{u}=(u_1,u_2,\dots,u_n)$ and  $\mathbf{v}=(v_1,v_2,\dots,v_n)$ in $\Fq^n$ is the number positions in which they differ, i.e., $d_H(\mathbf{u}, \mathbf{v})=|\{i: u_i\not = v_i \}|$. The minimum (Hamming) distance $d$ of $C$ is $d=\min \{d_H(\mathbf{u}, \mathbf{v}): \mathbf{u}, \mathbf{v} \in C, \mathbf{u}\not = \mathbf{v} \}$.   A code of length $n$, cardinality $M$, and minimum distance $d$ is denoted by $(n,M,d)$.  When $C$  is a vector subspace of $\Fq^n$, then it is a linear code, and the notation $[n,k,d]_q$ where $k$ is the dimension of $C$, is used the denote the basic parameters of $C$.  Most of the coding literature is on linear codes but we do not require a code to be linear in this work. 

One of the most important and challenging problems in coding theory is a discrete optimization problem: Given $q,n$ and $d$, what is the largest possible value of $|C|$ for a code $C$ of length $n$ over $\Fq$?  We denote this value by $A_q(n,d)$. A code  $C$ with  $|C|=A_q(n,d)$ is called an optimal code. In most cases, optimal codes are not known. In the case of linear codes, the online database \cite{codetables} provides information on best known linear codes (BKLC), including theoretical bounds and information about the constructions of the codes with currently best known parameters.

There are numerous bounds on the parameters of codes, some apply to all codes some only to linear codes. One of the most fundamental bounds is the Gilbert-Varshamov (GV) bound   which states that $\displaystyle{A_q(n,d)\geq \frac{q^n}{V(n,d-1)} \text{ where }  V(n,t)=\sum_{i=0}^t \binom{n}{i}(q-1)^i}$. A code that satisfies the inequality in the GV bound is called a GV code. An improvement of  the Gilbert-Varshamov bound for binary codes  was given in \cite{jiang2004}:

\begin{theorem}[Jiang, Vardy \cite{jiang2004}]
\label{thm: Jiang}
    Let $n$ and $d$ be positive integers with $d/n\leq0.499$. Then there exists a positive constant $c$ such that $$A_2(n,d)\geq c\frac{2^n}{V(n,d-1)}\log_2 V(n,d-1).$$
\end{theorem}

For small or specific prime-order alphabets the question of whether or not there exist codes of exponentially larger order than given by the Gilbert-Varshamov bound remains open and has received much attention. Interestingly, such families of codes have been found for sufficiently large, non prime-order alphabets. These constructions, called TVZ codes \cite{tvz1982},  rely on the underlying algebraic structure of large finite fields.

The goals of this work are twofold: the first is to consider the possible strength of local combinatorial statistics in finding bounds on the parameters of optimal codes. We are interested in whether or not it is possible to force exponentially larger-than-GV codes via global combinatorial parameters such as $H$-subgraph count, rather than algebraic structure. The hope is that this will illuminate what approaches will be ineffective in considering the optimality of the Gilbert-Varshamov bound for finite fields in which this problem remains open.
Secondly we want to outline a general structure for the study of codes over general metric spaces. The notion of a code over a metric space has seen multiple generalizations over time, and we are interested in the most general form. Thus we will also consider how combinatorial bounds can apply in the case of a general metric space.

We expand on standard asymptotic notation in order to consider the exponential gap: for sets of functions $A$ and $B$, a particular function $g$, and a function operation $\ast$, we  write $f\in g\ast A$ if there exists some $h\in A$ such that $f=g\ast h$, and we write $f\in A\ast B$ if there exists some $h\in A$ and $h'\in B$ such that $f=h\ast h'$. We use the notation $\mathbb R_{\geq0}$ for the set of nonnegative real-valued functions.

We mostly work in a graph-theoretic setting, as the question of finding optimal codes has a natural rephrasing in terms of finding independent sets in a specific graph.   All graphs considered in this paper are simple and finite. Given a graph $G$ with the vertex set $V(G)$ and the edge set $E(G)$, and any two vertices $u,v \in V(G)$, we write $u\sim v$  to mean that $u$ and $v$ are connected by an edge, i.e., they are incident, and $\{ u,v\} \in E(G)$. We also write $e=uv\in E(G)$. The number $|V(G)|$ of vertices of $G$ is called the order of $G$. We let $\Delta(G)$ and $\overline d(G)$ denote the maximum degree and average degree of a vertex in $G$ respectively. A set $W\subseteq V(G)$ is \textit{independent} if no edges lie between vertices in $W$. The \textit{independence} of $G$, $\alpha(G)$, is the order of the largest independent set in $G$. We call a pair of vertices $\{u,v\}\not\in E(G)$ a \textit{nonedge} for convenience. For graphs $H_1$ and $H_2$, we  define their \textit{join}, $H_1\vee H_2$ as the graph with vertices $V(H_1)\cup V(H_2)$ and edges $E(H_1)\cup E(H_2)\cup (V(H_1)\times V(H_2))$, and their \textit{disjoin union}, $H_1\sqcup H_2$ as the graph with vertices $V(H_1)\cup V(H_2)$ and edges $E(H_1)\cup E(H_2)$.

In this work, we consider a more general notion of a code in which the alphabet need not be a finite field or a ring.  We only need a space with a notion of distance to ask the question of interest. Thus we ask: for a metric space $M$, what is the largest subset $C$ of $M$ in which every pair in $C$ is at least $d$ apart? To this end, for a finite metric space $M$, we define $C\subseteq M$ to be a \textit{code of distance $d$} if $d(u,v)\geq d$ for all $u\not=v\in C$. We call elements of $C$ codewords. We  define \textit{$A(M,d)$} to be the size of the largest code of distance $d$ in $M$. A code $C$ of distance $d$ with $|C|=A(M,d)$ is called an \textit{optimal code}. Other notation will be introduced as needed throughout the paper.

\begin{remark}
\label{rmk: d}
In coding theory literature, it is standard to use $d$ to denote the minimum distance of a code. It is also standard in the context of metric spaces to use $d$ to refer to the underlying distance function. Moreover, in the context of graph theory, it is common to use $d$ to denote the average degree of a vertex. Therefore, the letter $d$ will have different meanings in different parts of this paper, reflecting its overloaded meanings in these fields.
\end{remark}


\noindent We recall Tur\'an's theorem in its standard form:

\begin{theorem}[Turán's Theorem {\cite{turan1941}}]
    Let $r\geq 1$. If $G$ is a $K_{r+1}$-free graph on $n$ vertices, then
\[
e(G)\leq e(T_r(n)),
\]
where $T_r(n)$ is the complete $r$-partite graph whose part sizes differ by at most one. Moreover, equality holds if and only if $G=T_r(n)$.
\end{theorem}

\noindent The following equivalent formulation will be applied multiple times throughout the paper.

\begin{theorem}[Reformulated Tur\'an’s Theorem]
    \label{thm: turan}
    Suppose $G$ is a graph with $n$ vertices and average degree $\overline d$. Then 
    \[
    \alpha(G)\geq\frac{n}{\overline{d}+1}.
    \]
    Equality holds if and only if $G$ is a disjoint union of cliques of equal order.
\end{theorem}

\section{Basic Results}

For any metric space $M$ with distance function $d$, we define the \textit{$d$-Proximity graph of $M$}, $\Gamma_{<d}(M)$ by $V(\Gamma_{<d}(M))=M$ and $u\sim v$ if and only if $d(u,v)<d$. For any $v\in M$, let $B_d(v)$ denote the open ball of radius $d$ around $v$. We  let $|B_d|=\frac{1}{|M|}\sum_{v}|B_d(v)|$. We observe the following two facts about proximity graphs:
\begin{remark}
\label{rmk: ind-code}
    For any finite metric space $M$, 
    \[
    \alpha(\Gamma_{<d}(M))=A(M,d).
    \]
\end{remark}

\begin{remark}
\label{rmk: degree-ball}
    For any finite metric space $M$ and for any $v\in M$, 
    \[
    \deg_{\Gamma_{<d}(M)}(v)=|B_d(v)|-1.
    \]
\end{remark}

\noindent Now, the classical GV lower bound can be stated in this  general setting as follows:
\begin{theorem}[Gilbert-Varshamov]
\label{thm: GV}
    For any finite metric space $M$, 
    \[
    A(M,d)\geq\frac{|M|}{|B_d|}.
    \]
\end{theorem}
\begin{proof}
This bound follows as an application of Tur\'an's theorem and remarks \ref{rmk: ind-code} and \ref{rmk: degree-ball}. 
\end{proof}

We call this theorem the generalized Gilbert-Varshamov bound in the more general setting of metric space codes. The original bound is due to Edgar Gilbert \cite{gilbert1952} and Rom Varshamov \cite{varshamov1957} who independently discovered it for the case where $M$ is $\Fq^n$ under the Hamming metric.

\begin{remark}
\label{rmk: GV char}
    Equality holds in the Gilbert-Varshamov bound exactly when $\Gamma_{<d}(M)$ is a degree-regular disjoint union of cliques.
\end{remark}
This follows as a direct corollary of the characterization of equivalence in Tur\'an's Theorem.

\section{Independence-Forcing Subgraph Counts}

Observing that codes in $M$ correspond to independent sets in the proximity graph of $M$ naturally leads us to consider what properties of a graph might force large independence. We thus consider the following Ramsey-Theoretic question: for a graph $G$ with few copies of $H$, what lower bounds on $\alpha(G)$ do we have? This question has received comparatively less attention than the well-know Graph-Ramsey numbers.

For two graphs $H$ and $G$, a \textit{homomorphism} from $H$ to $G$ is a function $g:H\to G$ such that if $\{u,v\}\in E(H)$ then $\{g(u),g(v)\}\in E(G)$. We  let $\hom(H,G)$ and $\inj(H,G)$ be the number of homomorphisms and injective homomorphisms from $H$ to $G$ respectively. We propose a plausible form for homomorphism-multiplicity Ramsey numbers.  This conjecture is informed by the results in this section.
\begin{conjecture}
\label{conj: Ramsey-Sidorenko}
    Given a graph $H$, there exist constants $a,C_H>0$ such that for any graph $G$ of order $n$,
    \[
    \alpha(G)\geq C_H\min\left\{n^{a},\left(\frac{n^{v(H)}}{\hom(H,G)}\right)^{1/e(H)}\right\}.
    \]
\end{conjecture}

We say that a graph $H$ is \textit{Ramsey-Sidorenko} if it satisfies the statement in Conjecture \ref{conj: Ramsey-Sidorenko} for all graphs $G$. In that case, we call the supremum of the constants $a(H)\leq 1$ such that $H$ satisfies the statement in Conjecture \ref{conj: Ramsey-Sidorenko} for all graphs $G$ the \textit{Ramsey-Sidorenko threshold} for $H$, denoted $a_0(H)$. The main result of this form comes from a relatively recent paper of Bohman-Mubayi.
\begin{theorem}[Bohman-Mubayi \cite{bohman2018}]
\label{thm: clique-independence}
    For each $s\geq2$, there exists some $c_s>0$ such that if a graph $G$ has order $n$,
\[
\alpha(G)\geq  c_s\min\left\{n^\frac{1}{s-1},\left(\frac{n^s}{\hom(K_s,G)}\right)^{\frac{1}{\binom{s}{2}}}\right\}.
\]
\end{theorem}
We deduce from this theorem that a clique of order $s$ is Ramsey-Sidorenko with threshold $a_0(K_s)\geq\frac{1}{s-1}$. A well-known conjecture of Sidorenko is relevant:
\begin{conjecture}[Sidorenko \cite{sidorenko1992}]
\label{conj: Sidorenko}
    For all bipartite graphs $H$ and all graphs $G$,
$$\hom(H,G)\geq v(G)^{v(H)}\left(\frac{2e(G)}{v(G)^2}\right)^{e(H)}.$$
\end{conjecture}
We say a bipartite graph $H$ is \textit{Sidorenko} if it satisfies this property for all graphs $G$. Various simple classes of graphs are known to be Sidorenko, such as trees, even cycles, and complete bipartite graphs. Recent results have greatly increased the class of graphs known to be Sidorenko (see \cite{conlon2018}).

\begin{remark}
    Sidorenko graphs $H$ are Ramsey-Sidorenko with $a_0(H)=1$.
\end{remark}

The above remark follows directly from Tur\'an's theorem. We can show that this property is closed under disjoint union and join. Together with the prior observation that many bipartite graphs are Ramsey-Sidorenko, these two results yield that many graphs are in fact Ramsey-Sidorenko.

\begin{proposition}
\label{prop: disconnected Ramsey-Sidorenko}
    If $H_1$ and $H_2$ are Ramsey-Sidorenko graphs, then $H_1\sqcup H_2$ is Ramsey-Sidorenko with $a_0(H_1\sqcup H_2)\leq\max\{a_0(H_1),a_0(H_2)\}$.
\end{proposition}
\begin{proof}
Let $H=H_1\sqcup H_2$ and $a<\max\{a_0(H_1),a_0(H_2)\}$. Observe $\hom(H,G)=\hom(H_1,G)\hom(H_2,G)$. Now, supposing there does not exist a $C_H>0$ such that $\alpha(G)< C_Hn^a$ we would then have for each $i=1,2$ that there exists a $C_{H_i}>0$ such that
\[
\alpha(G)\geq C_{H_i}\frac{n^{v(H_i)}}{\alpha(G)^{e(H_i)}}.
\]

As a result, letting $C_H=\min\{C_{H_i}:i\in[k]\}^{e(H)/k}$, we have
\begin{align*}
    \hom(H,G)&=\hom(H_1)\hom(H_2)\\
    &\geq C_{H_1}\frac{n^{v(H_1)}}{\alpha(G)^{e(H_1)}}\cdot C_{H_2}\frac{n^{v(H_2)}}{\alpha(G)^{e(H_2)}}\\
    &=C_H^{e(H)}\frac{n^{v(H)}}{\alpha(G)^{e(H)}},
\end{align*}

and thus $\alpha(G)\geq C_H\bigl(\frac{n^{v(H)}}{\hom(H,G)}\bigr)^{1/e(H)}$ as desired.
\end{proof}

\begin{proposition}
\label{prop: cones independence-forcing}
    If $H$ is Ramsey-Sidorenko, then $\overline{K_s}\vee H$ is Ramsey-Sidorenko with 
    \[
    a_0(\overline{K_s}\vee H)\geq\frac{a_0(H)}{1+sa_0(H)}.
    \]
\end{proposition}
\begin{proof}
Suppose $H$ is Sidorenko. For any positive integer $s$, let $J=\overline{K_s}\vee H$, let $a<a_0(H)$ and $b=\frac{a}{1+sa}$. Let $C_H>0$ be such that it together with $a$ satisfies conjecture \ref{conj: Ramsey-Sidorenko}. Consider any graph $G$ of order $n$ and independence $\alpha\leq\delta n^b$ for $\delta>0$ sufficiently small. More precisely, it suffices to take $\delta$ such that $2^s\left(\frac{2}{C_H}\right)^{1/a} \delta^{s+1/a}\leq 1/2$ for the sake of this argument.

Since we have assumed that $\alpha\leq\delta n^b$, our goal is to show that $\alpha$ lies in the random domain by applying Ramsey-Sidorenko-ness of $H$ on common neighborhoods of $s$ vertices. We let\\ $M=\sum_{X\in V(G)^s}|N_\cap(X)|$.

Observe\begin{align*}
    M&=\sum_{x\in V(G)}\deg(x)^s\geq n(\overline d(G))^s\text{ by convexity.}\\
    &\geq 2^{-s}\frac{n^{s+1}}{\alpha^s}\text{ by Tur\'an's theorem.}
\end{align*}

We want to consider $s$-sets of vertices with many shared neighbors:\\ Let $\mathcal L=\{X\in V(G)^s:|N_\cap(X)|\geq\left(\frac{2}{C_H}\right)^{1/a}\alpha^{1/a}\}$. We have
\begin{align*}
    \sum_{X\not\in\mathcal L}|N_\cap(X)|&\leq n^s\cdot\left(\frac{2}{C_H}\right)^{1/a}\alpha^{1/a}\\
    &\leq 2^s\left(\frac{2}{C_H}\right)^{1/a}\frac{M\alpha^{s+1/a}}{n}\text{ applying the previous bound on $M$.}\\
    &\leq 2^s\left(\frac{2}{C_H}\right)^{1/a}\cdot M\delta^{s+1/a}\text{ since we have $\alpha\leq\delta n^b$ therefore $\frac{\alpha^{s+1/a}}{n}\leq \delta^{s+1/a}$}.\\
    &\leq M/2\text{ as we chose $\delta$ such that $2^s\left(\frac{2}{C_H}\right)^{1/a} \delta^{s+1/a}\leq 1/2$}.
\end{align*}

As a result, we obtain $\sum_{X\in\mathcal L}|N_\cap(X)|\geq M/2$. Now for any $X\in\mathcal{L}$, 
\[
\hom(H,G[N_\cap(X)])\geq C_H^{e(H)}\frac{|N_\cap(X)|^{v(H)}}{\alpha^{e(H)}},
\]
so
\begin{align*}
    \hom(J,G)&\geq\sum_{X\in\mathcal{L}}\hom(H,G[N_\cap(X)])\\
    &\geq\frac{C_H^{e(H)}}{\alpha^{e(H)}}\sum_{X\in\mathcal L}|N_\cap(X)|^{v(H)}\\
    &\geq\frac{C_H^{e(H)}}{\alpha^{e(H)}}\cdot|\mathcal L|\left(\frac{M/2}{|\mathcal L|}\right)^{v(H)}\text{ by convexity.}\\
    &\geq\frac{C_H^{e(H)}}{\alpha^{e(H)}}\cdot\frac{(M/2)^{v(H)}}{(n^{s})^{v(H)-1}}\text{ taking }|\mathcal L|\leq n^s.\\
    &\geq C_H^{e(H)}2^{-(s+1)v(H)}\cdot\frac{n^{v(H)+s}}{\alpha^{e(H)+sv(H)}}\text{ again applying the lower bound on $M$.}\\
    &=C_H^{e(H)}2^{-(s+1)v(H)}\cdot\frac{n^{v(J)}}{\alpha^{e(J)}}.
\end{align*}

\noindent We thus obtain
\[
\alpha\geq C_J\left(\frac{n^{v(J)}}{\hom(J,G)}\right)^{1/e(J)}
\]
as desired for $C_J=C_H2^{-(s+1)v(H)/e(H)}$.
\end{proof}

We also wish to consider injection-multiplicity Ramsey numbers of a graph $H$. The following conjecture is informed again by the results in this section of the paper.
\begin{conjecture}
\label{conj: independence-forcing}
    Given a graph $H$, there exists a constant $h=h(H)<1$ such that for any $G$ of order $n$ with average degree $\overline d(G)\in n^{h+o(1)}$,
    \[
    \inj(H,G)\in n^{v(H)-(1-h)e(H)-\Omega(1)}\text{ implies } \alpha(G)\in n^{1-h+\Omega(1)}.
    \]
\end{conjecture}

We call a graph $H$ \textit{independence-forcing} if it satisfies the statement in Conjecture \ref{conj: independence-forcing} for all graphs $G$. In that case we  call the infimum constant $h(H)$ such that $H$ satisfies the statement in Conjecture \ref{conj: independence-forcing} for all graphs $G$ the \textit{independence-forcing density threshold} for $H$, denoted $h_0(H)$.

The next result shows that Ramsey-Sidorenko graphs are  independence forcing.

\begin{proposition}
    Ramsey-Sidorenko graphs $H$ are independence-forcing with 
    \[
    h_0(H)\leq\max\left\{1-a_0(H),\frac{1-e(H)}{e(H)}\right\}.
    \]
\end{proposition}
\begin{proof}
Fix a nontrivial Ramsey-Sidorenko graph $H$. Let $h>\max\left\{1-a_0(H),\frac{1-e(H)}{e(H)}\right\}$ and consider any graph $G$ of order $n$ and average degree $\overline d(G)\in n^{h+\Omega(1)}$. 
Since each homomorphism is either injective or identifies two vertices of $H$,
\[
\inj(H,G)+\binom{v(H)}2(n)_{v(H)-1}\geq\hom(H,G).
\]
    
We can thus take
\[
\max\{2\cdot \inj(H,G),2\binom{v(H)}2(n)_{v(H)-1}\}\geq\hom(H,G).
\]

Since $H$ is Ramsey-Sidorenko, we have:
\begin{align*}
    \alpha(G)&\geq C_H\min\left\{n^{a},\left(\frac{n^{v(H)}}{\hom(H,G)}\right)^{1/e(H)}\right\}\\
    &\geq C_H\min\left\{n^a,\left(\frac{n^{v(H)}}{2\binom{v(H)}2(n)_{v(H)-1}}\right)^{1/e(H)},\left(\frac{n^{v(H)}}{2\inj(H,G)}\right)^{1/e(H)}\right\}\\
    &\geq C_H\min\left\{n^a,n^{1/e(H)},\left(\frac{n^{v(H)}}{2\inj(H,G)}\right)^{1/e(H)}\right\}.
\end{align*}

Now if $\alpha(G)\geq C_H \min\{n^a,n^{1/e(H)}\}$, then
\[
\alpha(G)\geq C_H\min\{n^a,n^{1/e(H)}\}\in C_H n^{h+\Omega(1)}
\]
by our choice of $h$. 

On the other hand if $\alpha(G)\geq C_H\left(\frac{n^{v(H)}}{2\inj(H,G)}\right)^{1/e(H)}$, then
\[
\alpha(G)\in C_H\left(\frac{n^{v(H)}}{2n^{v(H)-(1-h)e(H)-\Omega(1)}}\right)^{1/e(H)}=n^{1-h+\Omega(1)}.
\]

In either case we have shown $\alpha(G)\in n^{1-h+\Omega(1)}$, so $H$ is independence-forcing with density threshold at most $\max\left\{1-a_0(H),\frac{1-e(H)}{e(H)}\right\}$.

\end{proof}

\noindent We have the following theorem on injective-copy counts of trees:

\begin{theorem}[Mubayi-Verstra\"ete \cite{mubayi2015trees}]
\label{thm: trees-density}
    Suppose a graph $G$ has order $n$, average degree $\overline d$, and $T_t$ is a tree on $t$ vertices. Then
$$\inj(T_t,G)\geq n\overline d(\overline d-t+2)^{t-2}.$$
\end{theorem}

A simple corollary of this result is that trees are vacuously independence-forcing with density threshold 0.

\section{Subgraph Counting For The Hamming Case}

In this section we show that for any fixed graph $H$, the proximity graph $\Gamma_{<\delta n}(\Fq^n)$ has asymptotically at least the random number of injections of $H$. In particular this shows that even if the Conjectures \ref{conj: Ramsey-Sidorenko} and \ref{conj: independence-forcing} are true for all graphs $H$, they  alone could not be used to find independent sets in $\Gamma_{<\delta n}(\Fq^n)$ (or codes in $\Fq^n$) beating the Gilbert-Varshamov lower bound.

We need to introduce more notation: For a probability distribution $\mu$ on $\Omega$, write
\[
\mathcal H_q(\mu)=-\sum_{x\in\Omega}\mu(x)\log_q\mu(x)
\]
for the base-$q$ entropy of $\mu$. We say a distribution has \textit{full support} if no element in the space is assigned probability 0. For a fixed $q$ and $0<\delta<1-1/q$, we  write 
\[
H_q(\delta)=-(1-\delta)\log_q(1-\delta)-\delta\log_q\left(\frac{\delta}{q-1}\right)
\]
for the normalized logarithmic density of a Hamming ball. We also use the convention that $0\log_q0=0$ for any prime power $q$.

The following is a standard specialization of the classical Karush--Kuhn--Tucker (KKT)  necessary conditions under Slater's constraint qualification (cf. \cite[Theorem L7.4]{farina2025kkt}).

\begin{theorem}[Karush--Kuhn--Tucker necessary conditions under Slater's condition]
\label{thm: kkt-slater}
Let \(U\subseteq\mathbb R^n\) be open and convex. Let $f,g_1,\dots,g_m:U\to\mathbb R$ be differentiable, and suppose that each \(g_i\) is convex. Let $h_1,\dots,h_r:U\to\mathbb R$ be affine.

Suppose that \(x^*\in U\) is a local minimizer of
\[
\begin{aligned}
\min\quad & f(x)\\
\text{subject to}\quad
& g_i(x)\leq 0,\qquad i=1,\dots,m,\\
& h_j(x)=0,\qquad j=1,\dots,r.
\end{aligned}
\]
Suppose further that Slater's condition holds: there exists
\(x^0\in U\) such that
\[
g_i(x^0)<0\qquad(i=1,\dots,m)\qquad\text{and}\qquad h_j(x^0)=0\qquad(j=1,\dots,r).
\]

Then there exist multipliers
\[
\lambda_1,\dots,\lambda_m\geq 0,
\qquad
\eta_1,\dots,\eta_r\in\mathbb R,
\]
such that:
\begin{enumerate}
\item[(i)] \emph{Primal feasibility:}
\[
g_i(x^*)\leq 0\qquad(i=1,\dots,m),
\qquad
h_j(x^*)=0\qquad(j=1,\dots,r).
\]

\item[(ii)] \emph{Dual feasibility:}
\[
\lambda_i\geq 0\qquad(i=1,\dots,m).
\]

\item[(iii)] \emph{Complementary slackness:}
\[
\lambda_i g_i(x^*)=0
\qquad(i=1,\dots,m).
\]

\item[(iv)] \emph{Stationarity:}
\[
\nabla f(x^*)
+\sum_{i=1}^m\lambda_i\nabla g_i(x^*)
+\sum_{j=1}^r\eta_j\nabla h_j(x^*)
=0.
\]
\end{enumerate}
\end{theorem}

Before the proof of the main theorem of this section we will need the following lemma:

\begin{lemma}
\label{lem: entropy-inequality}
Let $q$ be a prime power and $0<\delta<1-\frac1q$. For a given graph $H$, fix any vertex $r\in V(H)$ and let $\Omega=\{x:V(H)\to\Fq:x_r=0\}$. Then

\begin{align*}
    \Phi_{H,q}(\delta)&:= \max\left\{\mathcal H_q(\mu):\Pr_\mu(x_u\not=x_v)\leq\delta\text{ for every }e=uv\in E(H)\right\}\\
    &\ge v-1-m+mh_q(\delta).
\end{align*}

\end{lemma}

\begin{proof}
Throughout this argument we let $\varphi(x)=x\log_qx$.

First, our definition of $\Phi_{H,q}(\delta)$ as the maximum rather than the supremum is well-defined: Since the set of distributions $\mu$ satisfying $\Pr_\mu(x_u\not=x_v)\leq\delta$ for every $e\in E(H)$ forms a nonempty closed and bounded subset of the finite-dimensional space $\mathbb R^{|\Omega|}$ and base $q$ entropy is continuous, we can take a probability distribution $\mu$ on $\Omega$ attaining $\Phi_{H,q}(\delta)$.

Now we consider $\Phi_{H,q}$ as an optimization problem. Consider $\mu$ a vector in $\mathbb R^{|\Omega|}$, and let $a_{e}:=(\mathbf1_{\{x_u\not=x_v\}})_{x\in\Omega}$ for any $e=uv\in E(H)$. Then we find $-\Phi_{H,q}$ as
\[
\begin{aligned}
\min\quad & -\mathcal{H}_q(\mu)\\
\text{subject to}\quad
&\mathbf1^\top\mu-1=0,\\
&-\mu_x\leq0\qquad x\in\Omega,\\
&a_{e}^\top\mu-\delta\leq0\qquad e\in E(H)
\end{aligned}
\]
Each constraint is affine, and $-\mathcal{H}_q(\mu)$ is convex and differentiable over $\mu$ with full support, so we  show the optimal $\mu$ has full support in order to show we can use the KKT theorem:

Define $\nu$ coordinate-wise with $\nu_x=(1-\epsilon)\mathbf1_{\{x=\mathbf0\}}+\epsilon V_x$ where $V$ is the uniform distribution on $\Omega$ viewed as a vector in $\mathbb R^{|\Omega|}$. Then
\[
\Pr_{\nu}(x_u\not=x_v)=\epsilon\bigl(1-\frac 1q\bigr)<\delta\text{ choosing }0<\epsilon<\delta\frac{q}{q-1},
\]
and $\nu$ has full support: for any $x\in\Omega$, $\nu_x\geq\epsilon V_x>0$. Now we mix $\mu$ and $\nu$: let $\sigma(t)=(1-t)\mu+t\nu$ for $0<t<1$.

We now compare the contributions of $x$ yielding $\mu_x=0$ and $\mu_x>0$ to the difference in entropy of $\mu$ and $\sigma(t)$:

Let $P=\{x:\mu_x>0\}$ and consider $x\in P$. By the mean-value theorem,
\[
\varphi(\mu_x)-\varphi(\sigma(t)_x)=\varphi'(\beta_x(t))(\mu_x-\sigma(t)_x)=\varphi'(\beta_x(t))t(\mu_x-\nu_x)
\] 
for some $\beta_x(t)$ between $\mu_x$ and $\sigma(t)_x$. If $\nu_x\not=\mu_x$, we  let $t\leq\frac{\mu_x}{2|\nu_x-\mu_x|}$ for all of the finitely many $x\in P$ so that $\sigma(t)_x\in[\frac{\mu_x}2,\frac{3\mu_x}2]$, implying that
\[
|\varphi'(\beta(t)_x)|\leq\sup_{s\in[{\mu_x/}2,\ 3\mu_x/2]}\left|\varphi'(s)\right|<\infty.
\]
For the $x\in P$ with $\mu_x=\nu_x$ this follows immediately.

As a result, $\varphi(\mu_x)-\varphi(\sigma(t)_x)\leq C_xt$ for 
\[
C_x=|\mu_x-\nu_x|\sup_{s\in[{\mu_x/}2,\ 3\mu_x/2]}\left|\varphi'(s)\right|
\]
a constant with respect to $t$. 

Let $C=\sum_{x\in P}C_x$. Then for sufficiently small $t$,
\[
\sum_{x\in P}(\varphi(\sigma(t)_x)-\varphi(\mu_x))\geq-Ct.
\]

Now, if $x_0$ satisfies $\mu_{x_0}=0$, then
\begin{align*}
    \mathcal H_q(\sigma(t))-\mathcal H_q(\mu)&\geq\sum_{x\in P}\left(\varphi(\sigma(t)_x)-\varphi(\mu_x)\right)+\varphi(\sigma(t)_{x_0})\\
    &\geq-Ct+t\nu_{x_0}\log_q\frac{1}{t\nu_{x_0}}\\
    &\geq t(\nu_{x_0}\log_q\frac{1}{t}-C)\\
    &>0
\end{align*}
choosing $t$ sufficiently small.
This would therefore imply $\sigma(t)$ has higher entropy than $\mu$, and as $\sigma(t)$ is feasible, as it is a convex combination of feasible distributions, and $\mu$ is optimal, this cannot be the case. Thus, $\mu$ must have full support. We thus obtain the equivalent simplified optimization problem:

\[
\begin{aligned}
\min_{(0,\infty)^{|\Omega|}}\quad & -\mathcal{H}_q(\mu)\\
\text{subject to}\quad
&\mathbf1^\top\mu-1=0,\\
&a_{e}^\top\mu-\delta\leq0\qquad e\in E(H)
\end{aligned}
\]
Since the original optimizer \(\mu\) has full support, it is also a
local minimizer of this restricted problem. We also have shown $-\mathcal H_q$ is differentiable on an open neighborhood of $\mu$ contained in $(0,\infty)^{|\Omega|}$, and the distribution $\nu$ satisfies Slater's condition:
\[
\nu\in(0,\infty)^{|\Omega|},\qquad\mathbf1^\top\nu=1,\qquad a_e^\top\nu-\delta<0\qquad(e\in E(H)).
\]
We may therefore apply the ordinary finite-dimensional KKT multiplier
theorem, restricting the domain of $-\mathcal H_q$ to the open set $U:=(0,\infty)^{|\Omega|}$. Positivity is built into $U$, so the nonnegativity constraints need not be included. The KKT theorem yields numbers
\[
\lambda_e\geq0,\qquad\eta\in\mathbb R\qquad(e\in E(H))
\]
such that
\[
\lambda_e\bigl(a_e^\top\mu-\delta)=0\qquad (e\in E(H)),
\]
and
\[
\log_q(\mu_x)+\frac{1}{\ln q}+\sum_{e\in E(H)}\lambda_e(a_e)_x+\eta=0\qquad (x\in\Omega).
\]

\noindent Exponentiating,
\[
\mu_x=q^{-\sum_{e\in E(H)}\lambda_e(a_e)_x}q^{-\frac{1}{\ln q}-\eta}.
\]
Since 
\[
1=\sum_{x\in\Omega}\mu_x=q^{-\frac{1}{\ln q}-\eta}\left(\sum_{x\in\Omega}q^{-\sum_{e\in E(H)}\lambda_e(a_e)_x}\right),
\]
we can write 
\[
\mu_x=\frac{q^{-\sum_{e\in E(H)}\lambda_e(a_e)_x}}{Z(\lambda)},\ \text{where }Z(\lambda)=\sum_{x\in\Omega}q^{-\sum_{e\in E(H)}\lambda_e(a_e)_x}.
\]
Translating back to the logarithm we have
\[
-\log_q(\mu_x)=\sum_{e\in E(H)}\lambda_e(a_e)_x+\log_qZ(\lambda).
\]

We have found $\mu_x$ and $\log_q(\mu_x)$, so we may now find $\mathcal H_{q}(\mu)$:
\begin{align*}
    \mathcal H_{q}(\mu)&=-\sum_{x\in\Omega}\mu_x\log_q\mu_x\\
    &=\sum_{x\in\Omega}\mu_x\left(\sum_{e\in E(H)}\lambda_e(a_e)_x+\log_qZ(\lambda)\right)\\
    &=\left(\sum_{x\in\Omega}\mu_x\sum_{e\in E(H)}\lambda_e(a_e)_x\right)+\left(\sum_{x\in\Omega}\mu_x\cdot\log_q Z(\lambda)\right)\\
    &=\sum_{e\in E(H)}\lambda_ea_e^\top\mu+\log_qZ(\lambda)\\
    &=\delta\sum_{e\in E(H)}\lambda_e+\log_qZ(\lambda) \tag{1},
    \end{align*}
since by complementary slackness we have $\lambda_e(a_e^\top\mu-\delta)=0$ for each $e\in E(H)$.

Observe for any labeling $x$ and edge $e=uv$ we can write 
\[
(q^{-\lambda_e})^{\mathbf1_{x_u\not=x_v}}=q^{-\lambda_e}+(1-q^{-\lambda_e})\mathbf 1_{x_u=x_v},
\]

thus
\begin{align*}
    Z(\lambda)&=\sum_{x\in\Omega}\prod_{e\in E(H)}(q^{-\lambda_e}+(1-q^{-\lambda_e})\mathbf 1_{x_u=x_v})\\
    &=\sum_{x\in\Omega}\sum_{F\subseteq E(H)}\left((\prod_{e\in F}(1-q^{-\lambda_e})\mathbf 1_{x_u=x_v}))(\prod_{e\not\in F} q^{-\lambda_e})\right)\text{ by distribution laws.}\\
    &=\sum_{F\subseteq E(H)}\left((\prod_{e\in F}1-q^{-\lambda_e})(\prod_{e\not\in F} q^{-\lambda_e})\right)\sum_{x\in\Omega}\sum_{F\subseteq E(H)}\prod_{e\in F}\mathbf 1_{x_u=x_v}
\end{align*}

where
\[
\sum_{x\in\Omega}\sum_{F\subseteq E(H)}\prod_{e\in F}\mathbf 1_{x_u=x_v}=\sum_{F\subseteq E(H)}\sum_{x\in\Omega}\mathbf 1_{x_u=x_v\text{ for all }uv\in F}\geq\sum_{F\subseteq E(H)}q^{v(H)-|F|-1},
\]
since there are $q$ such choices for each of the $v(H)-1$ coordinates of $x$, except that coordinates in the same component of the graph defined by the edges of $F$ must be the same, and that graph surely has at most $|F|$ components.

We thus obtain
\begin{align*}
    Z(\lambda)&\geq\sum_{F\subseteq E(H)}\left(q^{v(H)-|F|-1}(\prod_{e\in F}1-q^{-\lambda_e})(\prod_{e\not\in F} q^{-\lambda_e})\right)\\
    &=q^{v(H)-1}\prod_{e\in E(H)}(q^{-\lambda_e}+\frac{1-q^{-\lambda_e}}{q})\text{ by distribution laws}.
\end{align*}

As a result we clearly have
\[
\log_q Z(\lambda)\geq v-1-m+\sum_e \log_q\bigl(1+(q-1)q^{-\lambda_e}\bigr),
\]

and so applying (1) we have shown
\[
\Phi_{H,q}(\delta)\geq v-1-m+\sum_e g_{q,\delta}(\lambda_e)\tag{2}
\]
where
\[
g_{q,\delta}(\lambda)=\log_q\bigl(1+(q-1)q^{-\lambda}\bigr)+\delta\lambda.
\]

We now minimize this one-variable function. Its derivative is
\[
g_{q,\delta}'(\lambda)=\delta-\frac{(q-1)q^{-\lambda}}{1+(q-1)q^{-\lambda}}.
\]

The unique critical point satisfies
\[
q^{-\lambda}=\frac{\delta}{(q-1)(1-\delta)}.
\]

Now since $\delta<1-\frac1q$, the right-hand side of the above expression is less than $1$, so this critical point lies in
$\lambda>0$. Furthermore,
\[
g_{q,\delta}''(\lambda)>0,
\]
so it is the minimum. At this point $\lambda_0$,
\[
1+(q-1)q^{-\lambda_0}=\frac1{1-\delta}
\]
and
\[
\lambda_0=\log_q\left(\frac{(q-1)(1-\delta)}{\delta}\right).
\]

\noindent Therefore, 
\begin{align*}
\min_{\lambda\geq0}g_{q,\delta}(\lambda)
&=-\log_q(1-\delta)+\delta\log_q\left(\frac{(q-1)(1-\delta)}{\delta}\right)\\
&=-(1-\delta)\log_q(1-\delta)-\delta\log_q\left(\frac{\delta}{q-1}\right)\\
&=H_q(\delta).
\end{align*}

\noindent Thus, every term in (2) is at least $H_q(\delta)$, proving
\[
\Phi_{H,q}(\delta)\geq v-1-m+mH_q(\delta).
\]

\noindent This completes the proof of the lemma.
\end{proof}

We can now prove the main theorem of the section.

\begin{theorem}
\label{thm: Hamming-H-Count}
For any fixed $q$, any $0<\delta\leq1-\frac{1}{q}$ and any fixed graph $H$,
\[
\inj\left(H,\Gamma_{<\delta n}(\Fq^n)\right)\in (q^n)^{v(H)-e(H)+e(H)h_q(\delta)+o(1)+\mathbb R_{\geq0}}.
\]
\end{theorem}

\begin{proof}
Fix some vertex $1\in V(H)$, and let
\[
\Omega=\{a\in\Fq^{V(H)}:a_1=0\}.
\]
We define random variables over the rows of
$M_{n\times(v(H)-1)}(\Fq)$ as follows. Let
\[
\varepsilon_n=n^{-1/3}
\qquad\text{and}\qquad
\delta_n=\delta-3\varepsilon_n.
\]
For sufficiently large $n$, we have $\delta_n>0$. 

We can take a probability distribution $\mu_n$ on $\Omega$ attaining $\Phi_{H,q}(\delta_n)$ in the statement of Lemma \ref{lem: entropy-inequality}. Thus, we have, for each
edge $i\ell\in E(H)$,
\[
\Pr_{\mu_n}(a_i\neq a_\ell)\leq\delta_n,
\]
and
\[
H_q(\mu_n)\geq
v(H)-1-e(H)+e(H)h_q(\delta_n).
\]

\noindent Referring to the uniform probability distribution on $\Omega$ by $U$, we
define
\[
\rho_n=(1-\varepsilon_n)\mu_n+\varepsilon_nU.
\]
Since entropy is concave and $U$ has maximum entropy, we have
\[
H_q(\rho_n)\geq H_q(\mu_n).
\]
Furthermore, for every $a\in\Omega$,
\[
\rho_n(a)\geq\frac{\varepsilon_n}{|\Omega|}.
\]
For each edge $i\ell\in E(H)$, we also have
\begin{align*}
\Pr_{\rho_n}(a_i\neq a_\ell)
&=(1-\varepsilon_n)\Pr_{\mu_n}(a_i\neq a_\ell)
+\varepsilon_n\Pr_U(a_i\neq a_\ell)\\
&\leq(1-\varepsilon_n)\delta_n
+\varepsilon_n\left(1-\frac1q\right)\\
&\leq\delta_n+\varepsilon_n\\
&=\delta-2\varepsilon_n.
\end{align*}

\noindent We now assign each $a\in\Omega$ to an integer via $\rho_n$: for each $n$, suppose
\[
\sum_{a\in\Omega}\lfloor n\rho_n(a)\rfloor=n-f(n).
\]
Since $\sum_{a\in\Omega}n\rho_n(a)=n$, we observe that $0\leq f(n)<|\Omega|$ treating $H$ and $q$ as fixed constants. Now we define a sequence $\{n_a\}_{a\in\Omega}$ with $n_a=\lfloor n\rho_n(a)\rfloor+1$ for $f(n)$ choices of $a$, and $n_a=\lfloor n\rho_n(a)\rfloor$ otherwise. Then
\[
\sum_{a\in\Omega}n_a=n
\qquad\text{and}\qquad
\frac{n_a}{n}\in\rho_n(a)+O\left(\frac1n\right)
\]
for each $a\in\Omega$. Furthermore,
\[
n\rho_n(a)\geq\frac{n^{2/3}}{|\Omega|}
\]
for every $a\in\Omega$. Thus, for sufficiently large $n$, we have $n_a\geq1$ for every $a\in\Omega$.

\noindent  Let
\[
T_n\subseteq M_{n\times(v(H)-1)}(\Fq)
\]
be the set of matrices in which each row pattern $(a_i)_{i\in V(H)\setminus\{1\}}$ occurs exactly $n_a$ times. Consider the $v(H)-1$ column vectors $x_2,x_3,\dots,x_{v(H)}\in\Fq^n$ and set $x_1=0$.

For any edge $1i\in E(H)$, we have
\begin{align*}
\frac{wt(x_i)}{n}
&=\sum_{a:a_i\neq0}\frac{n_a}{n}\\
&\in\sum_{a:a_i\neq0}\left(\rho_n(a)+O_{H,q}\left(\frac1n\right)\right)\\
&\subseteq\Pr_{\rho_n}(a_i\neq a_1)+O_{H,q}\left(\frac1n\right)\\
&\subseteq\delta-2\varepsilon_n+O_{H,q}\left(\frac1n\right).\\
\end{align*}
For sufficiently large $n$, and for any edge $i\ell\in E(H)$ with $i,\ell\neq1$, we have
\begin{align*}
\frac{d_H(x_i,x_\ell)}{n}
&=\sum_{a:a_i\neq a_\ell}\frac{n_a}{n}\\
&\in\sum_{a:a_i\neq a_\ell}\left(\rho_n(a)+O_{H,q}\left(\frac1n\right)\right)\\
&\subseteq\Pr_{\rho_n}(a_i\neq a_\ell)+O_{H,q}\left(\frac1n\right)\\
&\subseteq\delta-2\varepsilon_n+O_{H,q}\left(\frac1n\right),\\
\end{align*}
thus $wt(x_i),d_H(x_i,x_\ell)<\delta n$ for sufficiently large $n$.

We have therefore shown that $(x_1,x_2,\dots,x_{v(H)})$ is an ordered homomorphic copy of $H$ in $\Gamma_{<\delta n}(\Fq^n)$ for sufficiently large $n$.

We now show that this copy is injective. Fix distinct vertices
$i,\ell\in V(H)$. There exists some $a\in\Omega$ such that $a_i\neq a_\ell$ since $q\geq2$. Since $n_a\geq1$, this pattern occurs as at least one row of every matrix in $T_n$. Thus $x_i$ and $x_\ell$ differ in at least one coordinate, and therefore $x_i\neq x_\ell$.

We have thus shown that
\[
(x_1,x_2,\dots,x_{v(H)})
\]
is an ordered injective homomorphic copy of $H$ in $\Gamma_{<\delta n}(\Fq^n)$.

We now compute the order of $T_n$. By the multinomial formula, we have
\[
|T_n|=\binom{n}{\{n_a\}_{a\in\Omega}}=\frac{n!}{\prod_{a\in\Omega}n_a!}.
\]

Taking the logarithm,
\begin{align*}
\log_q|T_n|
&=\log_q\frac{n!}{\prod_{a\in\Omega}n_a!}\\
&=\log_qn!-\sum_{a\in\Omega}\log_qn_a!\\
&\in n\log_qn-\sum_{a\in\Omega}n_a\log_qn_a+O_{H,q}(\log n)
\qquad\text{by Stirling's formula}\\
&=-n\sum_{a\in\Omega}\frac{n_a}{n}
\log_q\left(\frac{n_a}{n}\right)+O_{H,q}(\log n).
\end{align*}

Writing $\nu_n(a)=\frac{n_a}{n}$, we thus obtain
\[
\log_q|T_n|=nH_q(\nu_n)+O_{H,q}(\log n).
\]

Since
\[
\nu_n(a)\in\rho_n(a)+O\left(\frac1n\right)
\]
for every $a\in\Omega$, we have
\[
H_q(\nu_n)\in H_q(\rho_n)+o(1).
\]
Therefore
\begin{align*}
H_q(\nu_n)
&\in H_q(\mu_n)+o(1)+\mathbb R_{\geq0}\\
&\subseteq v(H)-1-e(H)+e(H)h_q(\delta_n)+o(1)+\mathbb R_{\geq0}\\
&=v(H)-1-e(H)+e(H)h_q(\delta)+o(1)+\mathbb R_{\geq0}.
\end{align*}
Thus
\[
|T_n|\in q^{nH_q(\nu_n)+o(n)}\subseteq q^{n(v(H)-1-e(H)+e(H)h_q(\delta)+o(1)+\mathbb R_{\geq0})}.
\]

Different matrices in $T_n$ give different ordered injective homomorphisms,
since the matrix can be recovered from the vectors
\[
x_2,x_3,\dots,x_{v(H)}.
\]

Observing that translation by any element of $\Fq^n$ preserves
adjacency, we can repeat this process for each of the $q^n$ elements of
$\Fq^n$. Therefore
\[
\inj\left(H,\Gamma_{<\delta n}(\Fq^n)\right)\geq q^n|T_n|,
\]
and so
\[
\inj\left(H,\Gamma_{<\delta n}(\Fq^n)\right)\in (q^n)^{v(H)-e(H)+e(H)h_q(\delta)+o(1)+\mathbb R_{\geq0}}.
\]
\end{proof}

\noindent We note that when $H$ is not acyclic one can achieve the stronger result 
\[
\inj\left(H,\Gamma_{<\delta n}(\Fq^n)\right)\in (q^n)^{v(H)-e(H)+e(H)h_q(\delta)+\Omega(1)}.
\] 

We have achieved that every fixed subgraph appears in $\Gamma_{<\delta n}(\Fq^n)$ at exponentially at least the random rate. This is on its own an interesting observation about $\Fq^n$ under Hamming distance. What this implies is that a bound exponentially improving the order of codes over $\Fq$ cannot only utilize global statistics of small local structures in $\Fq^n$. In other words methods similar to those applied in Theorem \ref{thm: Jiang} cannot be used to find an exponential improvement. We also note that this result may be useful indirectly in upper bounding as well, as many classes of graphs with large independence have asymptotically few $H$ for some fixed $H$.

\section{Upper Bounds}

We also want to discuss combinatorial upper bounds on the size of codes. We first define some properties of spaces and graphs:

An \textit{isometry} on metric space $M$ is a permutation $i$ on $M$ such that $d(u,v)=d(i(u),i(v))$ for all $u,v\in M$. Hence, an isometry is a distance preserving bijection on a metric space. Then $M$ is \textit{homogeneous} if for any $u,v\in M$ there exists an isometry $i$ on $M$ such that $i(u)=v$. An \textit{automorphism} on graph $G$ is a permutation $f$ on $V(G)$ such that $u\sim v$ if and only $f(u)\sim f(v)$. $G$ is \textit{vertex-transitive} if for any $u,v\in V(G)$ there exists an automorphism $f$ on $G$ such that $f(u)=v$. Furthermore $G$ is \textit{nonedge-transitive} if for any $\{u_1,v_1\},\{u_2,v_2\}\not\in E(G)$ there exists an automorphism $f$ on $G$ such that $f(u_1)=u_2$ and $f(v_1)=v_2$. In other words $G$ is nonedge-transitive if $G^c$ is edge-transitive.

\begin{remark}
\label{rmk: homo->vtx trans}
    If $M$ is a homogeneous metric space, then $\Gamma_{<d}(M)$ is vertex-transitive.
\end{remark}

\begin{proposition}
\label{prop: homo->vtx in ind}
    If $G$ is vertex-transitive, then any vertex $v\in V(G)$ lies in some independent set of order $\alpha(G)$.
\end{proposition}
\begin{proof}
    Let $v\in V(G)$ and $C$ be an independent set of order $\alpha(G)$. Since $G$ is vertex-transitive, choosing any $c\in C$, there is a automorphism $\varphi$ such that $\varphi(c)=v$. Then $\varphi(C)$ is an independent set of order $\alpha(G)$: $\varphi(x)\not\sim \varphi(y)$ for any $x,y\in C$. Furthermore $|\varphi(C)|=|C|=\alpha(G)$ and $v\in \varphi(C)$.
\end{proof}

We begin with the trivial bound:

\begin{theorem}[Single-Ball Bound]
\label{thm: SB}
    For any metric space $M$,
    $$A(M,d)\leq |M|-|B_d|+1.$$
\end{theorem}
\begin{proof}
Suppose $C$ is an optimal code of distance $d$. Then choosing a vertex $v\in C$ with $|B_d(v)|\geq|B_d|$, $C\cap B_d(v)=\{v\}$, so elements of $M$ other than $v$ cannot be in both $C$ and $B_d(v)$. Therefore, $|M|\geq |C|+|B_d(v)|-1\geq|C|+|B_d|-1$. 
\end{proof}

We include this bound not for its tightness, but for the complement-like duality it shares with the Gilbert–Varshamov bound in the vertex-transitive  case of $\Gamma_{<d}(M)$.


\begin{proposition}
\label{prop: SB char}
For a homogeneous space $M$, equality holds in the Single-Ball bound exactly when $\Gamma_{<d}(M)$ is a multipartite graph.
\end{proposition}
\begin{proof}
Suppose $A(M,d)=|M|-|B_d|+1$, and let $u\not\sim v$ and $v\not\sim w$. By Proposition \ref{prop: homo->vtx in ind}, there is an independent set $C$ in $\Gamma_{<d}(M)$ containing $v$ with $|C|=n-|B_d|+1$, so $n=|C|+|B_d|-1$, therefore $M$ is partitioned by $C$ and $N(v)$ viewed in $\Gamma_{<d}(M)$. Thus, since $u,w\not\in N(v)$, it must be the case that $u,w\in C$, therefore $u\not\sim w$. Thus $\Gamma_{<d}(M)$ is multipartite. The other direction follows from counting the independence of degree-regular Tur\'an graphs.
\end{proof}

\noindent For stronger results we turn to independent set packings in graphs.

For a graph $G$, let $\mathcal{W}_k$ denote the set of independent sets of order $k$ in $G$. Then we  call a function $w:\mathcal{W}_k\to[0,1]$ a \textit{fractional packing} if for all vertices $v\in V(G)$, $\sum_{C\in\mathcal{W}_k:\,v\in C}w(C)\leq1$. We then define the \textit{fractional packing number} of $G$, $P_k^*(G)=\max\{\sum_{C\in\mathcal{W}_k}w(C):\text{$w$ is a fractional packing in $G$}\}$. We  call a function $w:\mathcal{W}_k\to[0,1]$ an \textit{edge-fractional packing} if for all nonedges $e\not\in E(G)$, $\sum_{C\in\mathcal{W}_k:\,e\subseteq C}w(C)\leq1$, and define the \textit{nonedge-fractional packing number} of $G$, $\mathcal{P}_k^*(G)=\max\{\sum_{C\in\mathcal{W}_k}w(C):\text{$w$ is an nonedge-fractional packing in $G$}\}$.

In the following two propositions we show some results about these packings:
\begin{proposition}
\label{prop: fractional}
    If $G$ is vertex-transitive, then $P_k^*(G)=\frac{n}{k}$  for any $k\leq\alpha(G)$.
\end{proposition}
\begin{proof}
    First, we will upper bound $P_k^*(G)$. To this end, suppose $w$ is an arbitrary fractional packing. Then $k\sum_{C\in\mathcal{W}_k}w(C)=\sum_{v\in V(G)}\sum_{C\in\mathcal{W}_k:\,v\in C}w(C)$ by double counting: each independent set $C\in\mathcal{W}_k$ will be counted by $k$ vertices in the right summand. Thus \begin{align*}
        \sum_{C\in\mathcal{W}_k}w(C)&=\frac{1}{k}\sum_{v\in V(G)}\sum_{C\in\mathcal{W}_k:\,v\in C}w(C)\\
        &\leq\frac{1}{k}\sum_{v\in V(G)}1\text{ since $w$ is a fractional packing.}\\
        &=\frac{n}{k},\text{  therefore $P_k^*(G)\leq\frac{n}{k}$.}
    \end{align*}

    Now we  provide a construction of a fractional packing to lower bound $P_k^*(G)$:

    Since $G$ is vertex-transitive, each vertex in $V(G)$ lies in some number  $r$ of independent sets of order $k$. Count the number of times that a vertex is found in an independent set in two ways: each independent set in $\mathcal{W}_k$ will be found by $k$ vertices, and each vertex will be in $r$ such independent sets, so by double counting we obtain $|\mathcal{W}_k|\cdot k=n\cdot r$. Choose $w(C)=\frac{1}{r}$ for each $C\in\mathcal{W}_k$. Then for each $v\in V(G)$, $\sum_{C\in\mathcal{W}_k:\,v\in C}w(C)=1$, so $w$ is a fractional packing. Now $\sum_{C\in\mathcal{W}_k} w(C)=|\mathcal{W}_k|\frac{1}{r}=\frac{n}{k}$, so $P_k^*(G)=\frac{n}{k}$.
\end{proof}

\begin{proposition}
\label{prop: nonedge-fractional}
    For any graph $G$ of order $n$ with average degree $\overline d$, for any $2\leq k\leq\alpha(G)$, $$\mathcal{P}_k^*(G)\le\frac{n(n-\overline d-1)}{k(k-1)}.$$ If $G$ is vertex-transitive, then $$\mathcal{P}_k^*(G)\ge\frac{n}{k}.$$ If $G$ is nonedge-transitive, then $$\mathcal{P}_k^*(G)=\frac{n(n-\overline d-1)}{k(k-1)}.$$
\end{proposition}
\begin{proof}
    Suppose $w$ is an arbitrary nonedge-fractional packing. Then \begin{align*}
        \binom{k}{2}\sum_{C\in\mathcal{W}_k}w(C)&=\sum_{e\not\in E(G):}\sum_{\substack{C\in\mathcal{W}_k\\ e\subseteq C}}w(C),\\
        &\text{ since each side counts $w(C)$ once for each nonedge in $C$.}\\
        &=\sum_{e\not\in E(G)}\sum_{C\ni e}w(C),\\
        &\leq\sum_{e\not\in E(G)}1=\binom{n}{2}-e(G).
    \end{align*}
    Thus $\sum_{C\in\mathcal{W}_k}w(C)\leq\frac{\binom{n}{2}-e(G)}{\binom{k}{2}}=\frac{n(n-\overline d-1)}{k(k-1)}$, and as $w$ was an arbitrary nonedge-fractional packing we have $\mathcal{P}_k^*(G)\le\frac{n(n-\overline d-1)}{k(k-1)}$.

    If we suppose $G$ is vertex-transitive, by proposition \ref{prop: fractional} we can take a fractional packing $w$ in $G$ with $\sum_{C\in\mathcal{W}_k}w(C)=\frac{n}{k}$. Now for any nonedge $e=\{a,b\}\not\in E(G)$, $\sum_{C\in\mathcal{W}_k:\,e\subseteq C}w(C)\leq\sum_{C\in\mathcal{W}_k:\,a\in C}w(C)\leq 1$ since $w$ is a fractional packing. Thus $w$ is also a nonedge-fractional packing and so $\mathcal{P}_k^*(G)\geq\frac{n}{k}$.

    If instead we suppose $G$ is nonedge-transitive, each nonedge in $G$ lies in some number $r$ of independent sets of order $k$. Count the number of times that a nonedge is found in an independent set in two ways: each independent set in $\mathcal{W}_k$ contains $\binom{k}{2}$ vertices, and each nonedge is in $r>0$ independent sets in $\mathcal{W}_k$, so by double counting we obtain $|\mathcal{W}_k|\cdot\binom{k}{2}=\left(\binom{n}{2}-e(G)\right)\cdot r$. Choose $w(C)=\frac{1}{r}$ for each $C\in\mathcal{W}_k$. Then for each $e\not\in E(G)$, $\sum_{C\in\mathcal{W}_k:\,e\subseteq C}w(C)=1$, so $w$ is a nonedge-fractional packing. Now $\sum_{C\in\mathcal{W}_k} w(C)=|\mathcal{W}_k|\frac{1}{r}=\frac{n(n-\overline d-1)}{k(k-1)}$, so $\mathcal{P}_k^*(G)=\frac{n(n-d-1)}{k(k-1)}$.
\end{proof}
We note that the vertex-transitive bound is optimal in that there are vertex-transitive graphs achieving $\mathcal{P}_k^*(G)=\frac{n}{k}$ even for $n-d>k$. We can now construct a bound on the independence number in the form of an upper bound on the number of triangles in the graph.

\begin{theorem}
\label{thm: K_3 UB by packing}
    For a graph $G$ of order $n$, independence $\alpha<n$, and regular degree $\Delta$, 
\[
3K_3(G)\leq n\binom{\Delta}{2}-\mathcal{P}_\alpha^*(G)\frac{\alpha\Delta(\alpha\Delta-(n-\alpha))}{2(n-\alpha)}.
\]
\end{theorem}
\begin{proof}
     Observe that there are exactly $\sum_{x\not\in C}\binom{|N(x)\cap C|}{2}$ induced 2-paths whose endpoints lie in $C$.
    Let $w$ be an nonedge fractional packing such that $\mathcal{P}_\alpha^*(G)=\sum_{C\in\mathcal{W}_\alpha}w(C)$.
    $$ind(P_2,G)\geq\sum_{C\in\mathcal{W}_\alpha(G)}\left(w(C)\sum_{x\not\in C}\binom{|N(x)\cap C|}{2}\right),$$
    since each induced 2-path with endpoints $u,v$ and middle vertex $w$ with is counted by the sum $\sum_{x\not\in C}\binom{|N(x)\cap C|}{2}$ for precisely those $C\in\mathcal{W}_\alpha$ containing $u$ and $v$, thus is counted with saturation at most $\sum_{C\supseteq\{u,v\}}w(C)\leq1$. So \begin{align*}
        ind(P_2,G)
        &\geq \sum_{C\in\mathcal{W}_\alpha(G)}\left(w(C)\cdot(\#x\not\in C)\binom{\frac{\sum_{x\not\in C}|N(x)\cap C|}{\#x\not\in C}}{2}\right)\text{ by convexity.}\\
        &=\sum_{C\in\mathcal{W}_\alpha(G)}\left(w(C)(n-\alpha)\binom{\frac{\alpha\Delta}{n-\alpha}}{2}\right)\text{ since:}
    \end{align*}
    $\sum_{x\not\in C}|N(x)\cap C|=\sum_{x\in C}|N(x)|=\sum_{x\in C}\Delta=\alpha\Delta,$ so \begin{align*}
        ind(P_2,G)&\geq\left(\sum_{C\in\mathcal{W}_\alpha}w(C)\right)(n-\alpha)\binom{\frac{\alpha\Delta}{n-\alpha}}{2}\\
        &=\mathcal{P}_\alpha^*(G)\frac{\alpha\Delta(\alpha\Delta-(n-\alpha))}{2(n-\alpha)}.
    \end{align*}

Observe that $ind(P_2,G)=\sum_{v\in V(G)}\left(\binom{\Delta}{2}-e(G[N(v)])\right)=n\binom{\Delta}{2}-3K_3(G)$, since there are a total of $\binom{\Delta}{2}$ possible edges in $G[N(v)]$, and each edge in $G[N(v)]$ forms a $K_3$ with $v$. Furthermore each 3-clique will be chosen exactly 3 times in this way: once for each vertex of the triangle. Thus we obtain
\[
3K_3(G)\leq n\binom{\Delta}{2}-\mathcal{P}_\alpha^*(G)\frac{\alpha\Delta(\alpha\Delta-(n-\alpha))}{2(n-\alpha)}.
\]

\end{proof}

\begin{corollary}
\label{cor: UB by tri_1}
    For a graph $G$ of order $n$, independence $\alpha<n$, and regular degree $\Delta$, if $G$ is vertex-transitive, then we have:
\[
6K_3(G)\leq
n\Delta^2\frac{n-2\alpha}{n-\alpha}.
\]

\noindent If $G$ is nonedge-transitive and $\alpha(G)\geq2$, then we have:
\[
6K_3(G)\leq n\Delta\left(\Delta-1-\frac{(n-\Delta-1)\bigl(\alpha(\Delta+1)-n\bigr)}{(\alpha-1)(n-\alpha)}\right).
\]
\end{corollary}
\begin{proof}

These results follow from applying Proposition \ref{prop: nonedge-fractional} and Theorem \ref{thm: K_3 UB by packing}. 

\end{proof}

In the nonedge-transitive case, under the additional constraint that $\alpha\Delta|\Delta-\alpha|=o(n^2)$ we remark that one can write $6K_3(G)\leq(1+o(1))n^2\frac{\Delta}{\alpha}$. In particular if $6K_3(G)\approx n\Delta^2$ one has $\alpha(G)\approx\frac{n}{\Delta}$.

\section{Concluding Remarks}

In this work, we have established a unified graph-theoretic framework for studying codes over general metric spaces. By identifying codes with independent sets in the proximity graph $\Gamma_{<d}(M)$, we  generalized the Gilbert-Varshamov bound and identified the conditions for  equality in the bound.

Our investigation into independence-forcing graphs reveals a significant limitation in the search for ``better-than-GV'' codes. Specifically, the results in Section 4 demonstrate that for any fixed graph $H$, the proximity graph of the Hamming space $\Fq^n$ contains at least the expected random number of copies of $H$. This implies that local combinatorial statistics -such as $H$-subgraph counts—cannot be used to force the existence of independent sets (codes) that are exponentially larger than those guaranteed by the GV bound.

This leads to a pivotal insight for coding theory: if codes beating the GV bound exist for small prime-order alphabets, they cannot be forced by local density alone. Instead, they must arise from the global, large-scale algebraic or combinatorial structure of the subset.

We end by listing some natural questions for further investigation in this area:
\begin{enumerate}

   \item  Ramsey-Sidorenko Classification: Determine if all graphs are Ramsey-Sidorenko and identify the general form of the threshold $a_0(H)$.
    \item Structured Subsets: Investigate whether Ramsey-Sidorenko type results can be applied to structured subsets of $\Fq^n$, rather than the entire space, to find improvements.
    \item Slowly Growing Families: Explore the independence-forcing phenomena for graph families $H$ whose order grows slowly with $n$.
\end{enumerate}

\newpage


\begin{thebibliography}{99}

\bibitem{turan1941}
P.~Tur\'an,
\newblock On an extremal problem in graph theory,
\newblock \emph{Matematikai \'es Fizikai Lapok} \textbf{48} (1941), 436--452.

\bibitem{jiang2004}
T.~Jiang and A.~Vardy,
\newblock Asymptotic improvement of the Gilbert--Varshamov bound on the size of binary codes,
\newblock \emph{IEEE Transactions on Information Theory} \textbf{50} (2004).
\newblock arXiv:math/0404325v1.

\bibitem{harris2019}
D.~G. Harris,
\newblock Some results on chromatic number as a function of triangle count,
\newblock \emph{SIAM Journal on Discrete Mathematics} \textbf{33} (2019), no.~1, 546--563.
\newblock \url{https://doi.org/10.1137/17M1112238}.

\bibitem{bohman2018}
T.~Bohman and D.~Mubayi,
\newblock Independence number of graphs with a prescribed number of cliques,
\newblock arXiv preprint arXiv:1801.01091, 2018.

\bibitem{mubayi2015trees}
D.~Mubayi and J.~Verstra\"ete,
\newblock The number of trees in a graph,
\newblock arXiv preprint arXiv:1511.07274, 2015.

\bibitem{gilbert1952}
E.~N. Gilbert,
\newblock A comparison of signalling alphabets,
\newblock \emph{Bell System Technical Journal} \textbf{31} (1952), no.~3, 504--522.
\newblock \url{https://doi.org/10.1002/j.1538-7305.1952.tb01393.x}.

\bibitem{varshamov1957}
R.~R. Varshamov,
\newblock Otsenka chisla signalov v kodakh s korrektsiei oshibok,
\newblock \emph{Doklady Akademii Nauk SSSR} \textbf{117} (1957), no.~5, 739--741.

\bibitem{conlon2018}
D.~Conlon, J.~H. Kim, C.~Lee, and J.~Lee,
\newblock Some advances on Sidorenko's conjecture,
\newblock \emph{Journal of the London Mathematical Society} \textbf{98} (2018), no.~3, 593--608.
\newblock \url{https://doi.org/10.1112/jlms.12142}.

\bibitem{sidorenko1992}
A.~F. Sidorenko,
\newblock Inequalities for functionals generated by bipartite graphs,
\newblock \emph{Discrete Mathematics and Applications} \textbf{2} (1992), no.~5, 489--504.
\newblock \url{https://doi.org/10.1515/dma.1992.2.5.489}.

\bibitem{farina2025kkt}
G.~Farina,
\newblock Lecture 7: Lagrange multipliers and KKT conditions,
\newblock MIT OpenCourseWare, 6.7220J/15.084 Nonlinear Optimization, 2025.
\newblock Lecture notes, Theorem L7.4.
\newblock \url{https://ocw.mit.edu/courses/6-7220j-nonlinear-optimization-spring-2025/resources/mit6_7220_s25_lec07_pdf/}.

\bibitem{codetables}
M.~Grassl,
\newblock Bounds on the minimum distance of linear codes,
\newblock \url{https://codetables.de}, 2026.
\newblock Accessed July 28, 2026.

\bibitem{tvz1982}
M.~A. Tsfasman, S.~G. Vl\u{a}du\c{t}, and Th.~Zink,
\newblock Modular curves, Shimura curves, and Goppa codes, better than the Varshamov--Gilbert bound,
\newblock \emph{Mathematische Nachrichten} \textbf{109} (1982), no.~1, 21--28.
\newblock \url{https://doi.org/10.1002/mana.19821090103}.

\end{thebibliography}
\end{document}